\documentclass[reqno, letterpaper, oneside]{article}
\usepackage{amsmath,amsfonts,amssymb,amsthm}
\usepackage[final]{graphicx}
\usepackage[usenames,dvipsnames]{color}

\usepackage{bbm}
\usepackage{setspace}
\usepackage{geometry}
\usepackage{enumerate}

\usepackage[colorlinks=false,,linktoc=allbookmarksopen=true,linktocpage,pdftex]{hyperref}
\usepackage[pdftex]{bookmark}

\newtheorem{theorem}{Theorem}
\newtheorem{lemma}[theorem]{Lemma}

\theoremstyle{definition}

\newcommand{\refT}[1]{Theorem~\ref{#1}}

\newcommand{\refL}[1]{Lemma~\ref{#1}}

\newcommand\E{\operatorname{\mathbb E{}}}
\renewcommand\Pr{\operatorname{\mathbb P{}}}

\newcommand\bigpar[1]{\bigl(#1\bigr)}

\def\rompar(#1){\textup(#1\textup)}    % usage: \rompar(...)

\def\xexp(#1){e^{#1}}
\newcommand\ceil[1]{\lceil#1\rceil}

\newcommand\floor[1]{\lfloor#1\rfloor}

\newcommand\noproof{\qed}

\newcommand{\eps}{\epsilon}

\newcommand{\cS}{\mathcal{S}}

\newcommand{\indic}[1]{\mathbbm{1}_{\{{#1}\}}}

\let\OLDthebibliography\thebibliography
\renewcommand\thebibliography[1]{
  \OLDthebibliography{#1}
  \setlength{\parskip}{0pt}
  \setlength{\itemsep}{0pt plus 0.3ex}
}

\title{Note on Sunflowers}
\author{Tolson Bell\thanks{School of Mathematics, Georgia Institute of Technology, Atlanta, GA~30332, USA.
E-mail: {\tt tbell37@gatech.edu}. Research supported by NSF Grant DMS-1851843.}, 
 {} Suchakree Chueluecha\thanks{Department of Mathematics, Lehigh University, Bethlehem, PA~18015, USA.
E-mail: {\tt suc221@lehigh.edu}. Research supported by Georgia Institute of Technology, College of Sciences. },
 and Lutz Warnke\thanks{School of Mathematics, Georgia Institute of Technology, Atlanta, GA~30332, USA.
E-mail: {\tt warnke@math.gatech.edu}. Research supported by NSF Grant DMS-1703516, NSF~CAREER grant~DMS-1945481 and a Sloan Research Fellowship.}}
\date{September~17, 2020; revised March~17, 2021}

\begin{document}
\maketitle

\begin{abstract}
A sunflower with $p$ petals consists of $p$~sets whose pairwise intersections are identical. 
The goal of the sunflower problem is to find the smallest $r = r(p,k)$ such that any family of~$r^k$ distinct $k$-element sets contains a sunflower with $p$ petals. 
Building upon a breakthrough of Alweiss, Lovett, Wu and Zhang from~2019, 
Rao proved that~$r=O(p \log(pk))$ suffices; this bound was reproved by Tao in~2020. 
In this short note we record that~$r=O(p \log k)$ suffices, 
by using a minor variant of the probabilistic part of these recent proofs.
\end{abstract}

\section{Introduction}
A sunflower with $p$ petals is a family of $p$ sets whose pairwise intersections are identical (the intersections may be empty). 
Let $\mathrm{Sun}(p,k)$ denote the smallest natural number~$s$ with the property that any family of at least~$s$ distinct $k$-element sets contains a \mbox{sunflower with $p$ petals}. 
In~1960, Erd{\H o}s and Rado~\cite{ErdosRado} proved that ${(p-1)^k} < {\mathrm{Sun}(p,k)} \le {(p-1)^k k! + 1} = O((pk)^k)$, 
and conjectured that for any $p \ge 2$ there is a constant~${C_p > 0}$ such that~${\mathrm{Sun}(p,k) \le C_p^k}$ for all~$k \ge 2$. 
This famous conjecture in extremal combinatorics 
was one of Erd{\H o}s' favorite problems~\cite{Erdos1981}, %fascinated Erd{\H o}s greatly~\cite{Erdos1981}, 
for which he offered a \$1000 reward~\cite{Erdos1990};  % for %proof or disproof of the conjecture
%for which he offered first~300~USD, then~500~USD and later~1000~USD, see~\cite{Erdos1976,Erdos1978,Erdos1990}, 
it remains open despite considerable attention~\cite{K}. 
%which remains open despite considerable attention~\cite{K} and a 500\$ monetary reward of Erd{\H o}s~\cite{Erdos1978}. 

In~2019, there was a breakthrough on the sunflower conjecture:
using iterative encoding arguments, \mbox{Alweiss}, Lovett, Wu and Zhang~\cite{ALWZ} proved that~${\mathrm{Sun}(p,k) \le (C p^3 \log k \log \log k)^k}$ for some \mbox{constant~$C > 0$}, 
opening the floodgates for further improvements.  
Using Shannon's noiseless coding \mbox{theorem}, Rao~\cite{RaoPreprint} subsequently simplified their proof and obtained a slightly better bound.
Soon thereafter, Frankston, Kahn, Narayanan and Park~\cite{FKNP} refined some key counting arguments from~\cite{ALWZ}.
Their ideas were then utilized by Rao~\cite{Rao} to improve the best-known sunflower bound to~${\mathrm{Sun}(p,k) \le (C p \log(pk))^k}$ 
for some \mbox{constant~$C>0$}, 
which in~2020 was reproved by Tao in his blog~\cite{Tao} using Shannon entropy~arguments. 
%(Several of these results exclude, for convenience, some simple boundary cases such as~$p=1$ or~$k=1$.) 

The aim of this short note is to record, for the convenience of other researchers, 
that a minor variant of (the probabilistic part of) the arguments from~\cite{Rao,Tao} 
gives~${\mathrm{Sun}(p,k) \le (C p \log k)^k}$ for some \mbox{constant~$C>0$}. 
%excluding for notational convenience the simple boundary cases $\mathrm{Sun}(p,1)=p$ and $\mathrm{Sun}(1,k)=1$. 
% (it was obtained at part of the 2020 REU at Georgia Tech). 
% 
\begin{theorem}\label{theorem:sunflower} 
There is a constant $C \ge 4$ such that $\mathrm{Sun}(p,k) \le (C p \log k)^k$ for all integers~$p,k \ge 2$.  
\end{theorem}
Setting~$r(p,k)=C  p \log k + \indic{k=1}p$, 
we shall in fact prove~${\mathrm{Sun}(p,k) \le r(p,k)^k}$ for all integers~$p \ge 2$ and~$k \ge 1$. 
Similarly to the strategy of~\cite{ALWZ,Rao,Tao}, this upper bound follows easily by induction on~$k \ge 1$ from \refL{lemma:main} below, 
where a family~$\cS$ of~$k$-element sets is called $r$-spread if there are at most $r^{k-|T|}$ sets of~$\cS$ that contain any non-empty set~$T$. 
(Indeed, the base case~$k=1$ is trivial due to~$r(p,1) = p$, 
and the induction step~$k \ge 2$ uses a simple case distinction: 
if~$\cS$ is~$r(p,k)$-spread, then \refL{lemma:main} guarantees a sunflower with $p$ petals; 
otherwise there is a non-empty set~$T$ such that more than $r(p,k)^{k-|T|} \ge r(p,k-|T|)^{k-|T|}$ sets of~$\cS$ contain~$T$, 
and among this family of sets we easily find a sunflower with $p$ petals using induction.)
\begin{lemma}\label{lemma:main}
There is a constant $C \ge 4$ such that, setting~$r(p,k) = C  p \log k$, 
the following holds for all integers~$p,k \ge 2$.  
If a family~$\cS$ with ${|\cS| \ge r(p,k)^k}$ sets of size~$k$ is \mbox{$r(p,k)$-spread}, then~$\cS$ contains~$p$ disjoint~sets. 
\end{lemma}
Inspired by~\cite{ALWZ}, in~\cite{Rao,Tao} probabilistic arguments 
are used to deduce \refL{lemma:main} with~$r(p,k) = \Theta(p \log(pk))$ from \refT{thm:petal} below, 
%(based on the union bound or linearity of expectation, respectively). 
where~$X_\delta$ denotes the random subset of~$X$ in which each element is included independently with probability~$\delta$. 
\begin{theorem}[Main technical estimate of~\cite{Rao,Tao}]\label{thm:petal} 
There is a constant~$B \ge 1$ such that the following holds for any integer~${k \ge 2}$, any reals~${0 <\delta,\epsilon \le 1/2}$, ${r \ge B\delta^{-1}\log(k/\epsilon)}$, and any family~$\cS$ of $k$-element subsets of a finite set~$X$. 
If~$\cS$ is \mbox{$r$-spread} with~$|\cS| \ge r^k$, then ${\Pr(\exists S \in \cS: S \subseteq X_\delta)} > {1-\epsilon}$.
\noproof  
\end{theorem}
\noindent
The core idea of~\cite{ALWZ,Rao,Tao} is to randomly partition the set~$X$ into $V_1 \cup \cdots \cup V_{p}$, 
by independently placing each element $x \in X$ into a randomly chosen~$V_i$. 
Note that the marginal distribution of each $V_i$ equals the distribution of~$X_{\delta}$ with~$\delta=1/p$. 
Invoking \refT{thm:petal} with~$\eps=1/p$ and~${r =  B\delta^{-1}\log(k/\epsilon)}$, 
a standard union bound argument implies that, with non-zero probability, 
all of the random partition-classes~$V_i$ contain a set from~$\cS$. 
Hence~$p$ disjoint sets~$S_1, \ldots, S_p \in \cS$ must~exist, 
which proves \refL{lemma:main} with~$r(p,k)= Bp \log(pk)$.

We prove \refL{lemma:main} with~$r(p,k)=\Theta(p \log k)$ using a minor twist: 
by randomly partitioning the vertex-set into more than~$p$ classes~$V_i$,  
and then using linearity of expectation (instead of a union bound). 
%(at the cost of slightly increasing the implicit~constant). 
% 
\begin{proof}[Proof of \refL{lemma:main}]% 
Set~$C = 4B$. 
We randomly partition the set~$X$ into $V_1 \cup \cdots \cup V_{2p}$, 
by independently placing each element $x \in X$ into a randomly chosen~$V_i$. 
Let~$I_i$ be the indicator random variable for the event that~$V_i$ contains a set from~$\cS$. %there is a set~$S \in \cS$ with~$S \subseteq V_i$. 
Since $V_i$ has the same distribution as~$X_{\delta}$ with~$\delta=1/(2p)$, 
by invoking \refT{thm:petal} with $\eps=1/2$ and~${r=r(p,k)} = {2Bp\log(k^2)} 
\ge {B\delta^{-1}\log(k/\epsilon)}$, we obtain~$\E I_i > 1/2$. 
Using linearity of expectation, the expected number of partition-classes~$V_i$ with $I_i=1$ is thus at least~$p$. 
Hence there must be a partition where at least~$p$ of the~$V_i$ contain a set from~$\cS$, 
which gives the desired $p$ disjoint sets~$S_1, \ldots, S_p \in \cS$. 
\end{proof}
\noindent
Generalizing this idea, \refT{thm:petal} gives~${p > \floor{1/\delta}(1-\eps)}$ disjoint sets~${S_1, \ldots, S_p \in \cS}$, 
which in the special case~${\floor{1/\delta} \eps \le 1}$ 
(used in~\cite{ALWZ,Rao,Tao} with~${\delta=\eps=1/p}$) 
simplifies to~${p\ge\floor{1/\delta}}$.

\section{Remarks}
Our proof of \refL{lemma:main} only invokes \mbox{\refT{thm:petal}} with~$\eps=1/2$, % (instead of~$\eps=1/p$ as in previous work). 
i.e., it does not exploit the fact that \mbox{\refT{thm:petal}} has an essentially optimal dependence on~$\eps$ (see \refL{lem:sharp} below). 
In particular, this implies that we could alternatively also 
prove \refL{lemma:main} and thus the~$\mathrm{Sun}(p,k) \le (C p \log k)^k$ bound of \mbox{\refT{theorem:sunflower}} 
using the combinatorial arguments of Frankston, Kahn, Narayanan and Park~\cite{FKNP} 
(we have verified that the proof of~\mbox{\cite[Theorem~1.7]{FKNP}} can be extended 
to yield \mbox{\refT{thm:petal}} under the stronger assumption~${r \ge B \delta^{-1} \max\{\log k, \log^2 (1/\eps)\}}$, say). 

We close by recording that %the estimate of 
\refT{thm:petal} is essentially best possible with respect to the $r$-spread assumption, 
which follows from the construction in~\cite[Section~2]{ALWZ} 
that in turn builds upon~\cite[Theorem~II]{ErdosRado}. 
\begin{lemma}\label{lem:sharp}
For any reals ${0 <\delta,\epsilon \le 1/2}$ and any integers ${k \ge 1}$, ${1 \le r \le 0.25 \delta^{-1}\log(k/\eps)}$, 
there exists an $r$-spread family~$\cS$ of $k$-element subsets of~$X=\{1, \ldots, rk\}$ 
with $|S| = r^k$ and ${\Pr(\exists S \in \cS: S \subseteq X_\delta)} < {1-\epsilon}$. 
\end{lemma}
\begin{proof}%  
We fix a partition~$V_1 \cup \cdots \cup V_{k}$ of~$X$ into sets of equal size~$|V_i| = r$, 
and define~$\cS$ as the family of all \mbox{$k$-element} sets containing exactly one element from each~$V_i$. 
It is easy to check that~$\cS$ is \mbox{$r$-spread}, with~$|\cS|=r^k$. 
Focusing on the necessary event that~$X_\delta$ contains at least one element from each~$V_i$, 
we~obtain 
\begin{align*}
\Pr(\exists S \in \cS: S \subseteq X_\delta) 
\: \le \: 
\bigpar{1-(1-\delta)^r}^k \le e^{-(1-\delta)^r k} < e^{-e^{-2\delta r}k} \le e^{-\sqrt{\eps k}} < 1-\eps 
\end{align*}
by elementary considerations (since~$e^{-\sqrt{\eps}} < 1-\eps$ due to~$0 < \eps \le 1/2$). 
\end{proof}

\section{Acknowledgements}
This research was conducted as part of the 2020~REU program at Georgia Institute of Technology.  

\pagebreak[3]

\small
\bibliographystyle{amsplain}

\pagebreak[3]
\normalsize

\begin{appendix}

\section*{Appendix: \refT{thm:petal}}\label{app}
\refT{thm:petal} follows from Tao's proof of Proposition~5 in~\cite{Tao} %(see below Corollary~7 therein) 
(noting that any~$r$-spread family~$\cS$ with $|\cS| \ge r^k$ sets of size~$k$ is also $r$-spread in the sense of~\cite{Tao}).
We now record that \refT{thm:petal} also follows from Rao's proof of~Lemma~4 in~\cite{Rao} 
(where the random subset of~$X$ is formally chosen in a slightly different~way). %(with a fixed number of elements). 
\begin{proof}[Proof of \refT{thm:petal} based on~\cite{Rao}]% 
Set~$\gamma = \delta/2$ and~$m = \ceil{\gamma |X|}$. 
Let~$X_i$ denote a set chosen uniformly at random from all $i$-element subsets of~$X$. 
Since $X_\delta$ conditioned on containing exactly~$i$ elements has the same distribution as~$X_i$, 
by the law of total probability and monotonicity it routinely follows that 
${\Pr(\exists S \in \cS: S \subseteq X_\delta)} $ is at least ${\Pr(\exists S \in \cS: S \subseteq X_m)} \cdot {\Pr(|X_\delta| \ge m)}$. 
The proof of~Lemma~4 in~\cite{Rao} %(see the beginning of Section~4 therein) 
shows that ${\Pr(\exists S \in \cS: S \subseteq X_m)} > 1-\eps^2$ whenever~$r \ge \alpha \gamma^{-1}\log(k/\epsilon)$, where $\alpha > 0$ is a sufficiently large constant. 
Noting ${|\cS| \le |X|^k}$ we see that ${|\cS| \ge r^k}$ enforces ${|X| \ge r}$, 
so standard Chernoff bounds (such as~\cite[Theorem~2.1]{JLR}) 
imply that ${\Pr(|X_\delta| < m)} \le {\Pr(|X_\delta| \le |X|\delta/2)}$ is at most $e^{-|X|\delta/8} \le e^{-r \delta/8} \le \eps^2$ whenever~$r \ge {16 \delta^{-1}\log(1/\epsilon)}$. 
This completes the proof with~$B = \max\{2\alpha,16\}$, say (since~$(1-\eps^2)^2 \ge 1-\eps$ due to~$0 < \eps \le 1/2$). 
\end{proof}
\end{appendix}

\end{document}